\documentclass[onefignum,onetabnum]{siamart171218}

\usepackage{lipsum}
\usepackage{amssymb}
\usepackage{amsfonts}
\usepackage{amsmath}

\usepackage{graphicx,epstopdf} 
\usepackage[caption=false]{subfig}
\usepackage[a4paper]{geometry}
\ifpdf
  \DeclareGraphicsExtensions{.eps,.pdf,.png,.jpg}
\else
  \DeclareGraphicsExtensions{.eps}
\fi

\usepackage{amsopn}

\newcommand{\be}{\begin{equation}}
\newcommand{\ee}{\end{equation}}
\newcommand{\bi}{\begin{itemize}}
\newcommand{\ei}{\end{itemize}}
\newcommand{\norm}[1]{\left\Vert{#1}\right\Vert} 
\newcommand{\ps}[2]{\langle #1,#2\rangle}

\newsiamremark{remark}{Remark}

\newcommand{\R}{\mathbb{R}}
\newcommand{\HH}{\mathcal{H}}
\newcommand{\hu}{\textbf{H}_1}
\newcommand{\hd}{\textbf{H}_2}

\newcommand{\f}{\alpha}
\newcommand{\EE}{\mathcal{E}}

\def\argmin{\textup{argmin}\,}
\newcommand{\off}[1]{}

\title{Optimal convergence rates for Nesterov acceleration}
\date{\today}
\author{J-F. Aujol $^1$, Ch. Dossal $^2$ and A. Rondepierre $^{2,3}$\\
\mbox{} \\
$^1$ Univ. Bordeaux, Bordeaux INP, CNRS, IMB, UMR 5251, F-33400 Talence, France.\\
$^2$ IMT, Univ. Toulouse, INSA Toulouse, France. \\
$^3$ LAAS, Univ. Toulouse, CNRS, Toulouse, France.
\mbox{} \\
{\scriptsize Jean-Francois.Aujol@math.u-bordeaux.fr,
$\{$Charles.Dossal,Aude.Rondepierre$\}$@insa-toulouse.fr}
}

\begin{document}

\maketitle

\begin{abstract}
In this paper, we study the behavior of solutions of the ODE associated to Nesterov acceleration. It is well-known since the pioneering work of Nesterov that the rate of convergence $O(1/t^2)$ is optimal for the class of convex functions with Lipschitz gradient.
In this work, we show that better convergence rates can be obtained with some additional geometrical conditions, such as \L ojasiewicz property. More precisely, we prove the optimal convergence rates that can be obtained depending on the geometry of the function $F$ to minimize. The convergence rates are new, and they shed new light on the behavior of Nesterov acceleration schemes. We prove in particular that the classical Nesterov scheme may provide convergence rates that are worse than the classical gradient descent scheme on sharp functions: for instance, the convergence rate for strongly convex functions is not geometric for the classical Nesterov scheme (while it is the case for the gradient descent algorithm). This shows that applying the classical Nesterov acceleration on convex functions without looking more at the geometrical properties of the objective functions may lead to sub-optimal algorithms.
\end{abstract}


 
\section{Introduction}\label{sec_intro}
The motivation of this paper lies in the minimization of a differentiable function $F:\R^n\rightarrow \R$ with at least one minimizer. Inspired by Nesterov pioneering work \cite{nesterov1983method}, we study the following ordinary differential equation (ODE):  
\begin{equation} \label{ODE}
\ddot{x}(t)+\frac{\f}{t}\dot{x}(t)+\nabla F(x(t))=0,
\end{equation}
where 
$\f>0$, 
with $t_0>0$, $x(t_0)=x_0$ and $\dot x(t_0)=v_0$. 
This ODE is associated to the Fast Iterative Shrinkage-Thresholding Algorithm (FISTA)\cite{beck2009fast} or the Accelerated Gradient Method \cite{nesterov1983method} : 
\begin{equation}
x_{n+1}=y_n-h \nabla F(y_n) \text{ and } 
y_n=x_n+\frac{n}{n+\f}(x_n-x_{n-1}),
\end{equation}
with $h$ and $\f$ positive parameters. 
This equation, including or not a perturbation term, has been widely studied in the literature \cite{attouch2000heavy,su2016differential,cabot2009long,balti2016asymptotic,may2015asymptotic}. This equation belongs to a set of similar equations with various viscosity terms. It is impossible to mention all works related to the heavy ball equation or other viscosity terms. We refer the reader to the following recent works \cite{Begout2015,jendoubi2015asymptotics,may2015asymptotic,cabot2007asymptotics,attouch2002dynamics,polyak2017lyapunov,attouch2017asymptotic} and the references therein.

Throughout the paper, we assume that, for any initial conditions $(x_0,v_0)\in \R^n\times\R^n$, the Cauchy problem associated with the differential equation \eqref{ODE}, has a unique global solution $x$ satisfying $(x(t_0),\dot x(t_0) )=(x_0,v_0)$. This is guaranteed for instance when the gradient function $\nabla F$ is Lipschitz  on bounded subsets of $\R^n$. 

In this work we investigate the convergence rates of the values $F(x(t))-F^*$ for the trajectories of the ODE \eqref{ODE}. It was proved in \cite{attouch2018fast} that if $F$ is convex with Lipschitz gradient and if $\f>3$, the trajectory $F(x(t))$ converges  to the minimum $F^*$ of $F$. It is also known that for $\f\geqslant  3$ and $F$ convex we have:
\begin{equation}
F(x(t))-F^*=O \left(
t^{-2}
\right).
\end{equation}   

Extending to the continuous setting the work of Chambolle-Dossal~\cite{chambolle2015convergence} of the convergence of iterates of FISTA, Attouch et al. \cite{attouch2018fast} proved that for 
$\f>3$ the trajectory $x$ converges (weakly in infinite-dimensional Hilbert space) to a minimizer of $F$. 
Su et al. \cite{su2016differential} proposed some new results, proving the integrability of $t\mapsto t(F(x(t))-F^*)$ when $\f>3$, and they gave more accurate 
bounds on $F(x(t))-F^*$ in the case of strong convexity. Always in the case of the strong convexity of $F$, Attouch,   Chbani, Peypouquet and Redont proved in \cite{attouch2018fast} that the trajectory $x(t)$ satisfies $F(x(t))-F^*=O\left( t^{-\frac{2\alpha}{3}}\right)$ for any $\alpha>0$. More recently several studies including a perturbation term \cite{attouch2018fast,AujolDossal,aujol2015stability,vassilis2018differential} have been proposed.

In this work, we focus on the decay of $F(x(t))-F^*$ depending on more general geometries of $F$ around its set of minimizers than strong convexity.  
Indeed, Attouch et al. in \cite{attouch2018fast} proved that if $F$ is convex then for any $\alpha>0$, $F(x(t))-F^*$ tends to $0$ when $t$ goes to infinity. Combined with the coercivity of $F$, this convergence implies that the distance
$d(x(t),X^*)$ between $x(t)$ and the set of minimizers $X^*$ tends to $0$. To analyse the asymptotic behavior of $F(x(t))-F^*$  we can thus only assume hypotheses on $F$ only on the neighborhood of $X^*$ and may avoid the tough question of the convergence of the the trajectory $x(t)$ to a point of $X^*$. 

More precisely, we consider functions behaving like $\norm{x-x^*}^{\gamma}$ around their set of minimizers for any $\gamma\geqslant 1$. Our aim is to show the optimal convergence rates that can be obtained depending on this local geometry. In particular we prove that if $F$ is strongly convex with a Lipschitz continuous gradient, the decay is actually better than $O\left( t^{-\frac{2\alpha}{3}}\right)$. We also prove that the actual decay for quadratic functions is  $O\left(
t^{-\alpha}
\right)$. These results rely on two geometrical conditions: a first one ensuring that the function is sufficiently flat around the set of minimizers, and a second one ensuring that it is sufficiently sharp. In this paper, we will show that both conditions are important to get the expected convergence rates: the flatness assumption ensures that the function is not too sharp and may prevent from bad oscillations of the solution, while the sharpness condition ensures that the magnitude of the gradient of the function is not too low in the neighborhood of the minimizers.

The paper is organized as follows.
In Section~\ref{sec_geom}, we introduce the geometrical hypotheses we consider on the function $F$, and their relation with \L ojasiewicz property. We then recap the state of the art results on the ODE \eqref{ODE} in Section~\ref{sec_state}. We present the contributions of the paper in Section~\ref{sec_contrib}: depending on the geometry of the function $F$ and the value of the damping parameter $\alpha$, we give optimal rates of convergence. The proofs of the theorems are given in Section~\ref{sec_proofs}. Some technical proofs are postponed to Appendix~\ref{appendix}.

\section{Local geometry of convex functions}\label{sec_geom}
Throughout the paper we assume that the ODE \eqref{ODE} is defined in $\R^n$ equipped with the euclidean scalar product $\langle \cdot,\cdot\rangle$ and the associated norm $\|\cdot\|$. As usual $B(x^*,r)$ denotes the open euclidean ball with center $x^*$ and radius $r>0$ while $\bar B(x^*,r)$ denotes the closed euclidean ball with center $x^*$ and radius $r>0$.

In this section we introduce two notions describing the geometry of a convex function around its minimizers.
\begin{definition}
Let $F:\R^n\rightarrow \R$ be a convex differentiable function, $X^*:=\argmin F\neq \emptyset$ and: $F^*:=\inf F$. 
\begin{enumerate}
\item[(i)] Let $\gamma \geqslant 1$. The function $F$ satisfies the hypothesis $\hu(\gamma)$ if, for any minimizer $x^*\in X^*$, there exists $\eta>0$ such that:
$$\forall x\in B(x^*,\eta),\quad F(x) - F^* \leqslant \frac{1}{\gamma} \langle \nabla F(x),x-x^*\rangle.$$
\item[(ii)] 
Let $r\geqslant 1$. The function $F$ satisfies the growth condition $\hd(r)$ if, for any minimizer $x^*\in X^*$, there exist $K>0$ and $\varepsilon>0$, such that:
\begin{equation*}
\forall x\in B(x^*,\varepsilon),\quad K d(x,X^*)^{r}\leqslant F(x)-F^*.
\end{equation*}
\end{enumerate}
\end{definition}


The hypothesis $\hu(\gamma)$ has already been used in \cite{cabot2009long} and later in \cite{su2016differential,AujolDossal}. This is a mild assumption, requesting slightly more than the convexity of $F$ in the neighborhood of its minimizers. Observe that any convex function automatically satisfies $\hu(1)$ and that any differentiable function $F$ for which $(F-F^*)^{\frac{1}{\gamma}}$ is convex for some $\gamma\geq 1$, satisfies $\hu(\gamma)$. Nevertheless having a better intuition of the geometry of convex functions satisfying $\hu(\gamma)$ for some $\gamma\geq 1$, requires a little more effort:
\begin{lemma}
Let $F:\R^n \rightarrow \R$ be a convex differentiable function with $X^*=\argmin F\neq \emptyset$, and $F^*=\inf F$. If $F$ satisfies $\hu(\gamma)$ for some $\gamma\geq 1$, then:
\begin{enumerate}
\item $F$ satisfies $\hu(\gamma')$ for all $\gamma'\in[1,\gamma]$.
\item For any minimizer $x^*\in X^*$, there exists $M>0$ and $\eta >0$ such that:
\begin{equation}
\forall x\in B(x^*,\eta),~F(x) -F^* \leqslant M \|x-x^*\|^\gamma.\label{hyp:H1}
\end{equation}
\end{enumerate} \label{lem:geometry}
\end{lemma}
\begin{proof}
The proof of the first point of Lemma \ref{lem:geometry} is straightforward. The second point relies on the following elementary result in dimension $1$: let $g:\R\rightarrow \R$ be a convex differentiable function such that $0\in \argmin g$, $g(0)=0$ and:
$$\forall t\in [0,1],~g(t) \leq \frac{t}{\gamma}g'(t),$$ 
for some $\gamma\geqslant 1$. Then the function $t\mapsto t^{-\gamma}g(t)$ is monotonically increasing on $[0,1]$ and:
\begin{equation}
\forall t\in [0,1],~g(t)\leqslant g(1)t^\gamma.\label{majo1D}
\end{equation}

Consider now any convex differentiable function $F:\R^n \rightarrow \R$ satisfying the condition $\hu(\gamma)$, and $x^*\in X^*$. There then exists $\eta>0$ such that:
$$\forall x\in B(x^*,\eta),\quad 0\leqslant F(x) - F^* \leqslant \frac{1}{\gamma} \langle \nabla F(x),x-x^*\rangle.$$
Let $\eta'\in(0,\eta)$. For any $x\in\bar B(x^*,\eta')$ with $x \neq x^*$, we introduce the following univariate function:
$$g_x:t\in [0,1]\mapsto F\left(x^*+t\eta'\frac{x-x^*}{\|x-x^*\|}\right)-F^*.$$
First observe that, for all $x\in \bar B(x^*,\eta')$ with $x \neq x^*$ and for all $t\in [0,1]$, we have: $x^*+t\eta'\frac{x-x^*}{\|x-x^*\|}\in \bar B(x^*,\eta').$ Since $F$ is continuous on the compact set $\bar B(x^*,\eta')$, we deduce that:
\begin{equation}
\exists M>0,~\forall x\in \bar B(x^*,\eta')  \ \mbox{ with $x \neq x^*$},~ \forall t\in[0,1],~g_x(t)\leq M.\label{bounded}
\end{equation}
Note here that the constant $M$ only depends on the point $x^*$ and the real constant $\eta'$. 

Then, by construction, $g_x$ is a convex differentiable function satisfying: $0\in \argmin(g_x)$, $g_{x}(0)=0$ and:
\begin{eqnarray*}
\forall t\in (0,1],~g_x'(t) &=& \left\langle \nabla F\left(x^*+t\eta'\frac{x-x^*}{\|x-x^*\|}\right),\eta'\frac{x-x^*}{\|x-x^*\|}\right\rangle\\
&\geqslant & \frac{\gamma}{t}\left(F\left(x^*+t\eta\frac{x-x^*}{\|x-x^*\|}\right)-F^*\right) = \frac{\gamma}{t} g_x(t).
\end{eqnarray*}
Thus, using the one dimensional result \eqref{majo1D} and the uniform bound \eqref{bounded}, we get:
\begin{equation}
\forall x\in \bar B(x^*,\eta') \ \mbox{ with $x \neq x^*$},~\forall t\in [0,1],~g_{x}(t) \leqslant g_x(1)t^{\gamma}\leqslant Mt^{\gamma}.
\end{equation}
Finally by choosing $t=\frac{1}{\eta'}\|x-x^{\ast}\|$, we obtain the expected result.
\end{proof}

In other words, the hypothesis $\hu(\gamma)$ can be seen as a ``flatness'' condition on the function $F$ in the sense that it ensures that $F$ is sufficiently flat (at least as flat as $x\mapsto \|x\|^\gamma$) in the neighborhood of its minimizers.

The hypothesis $\hd(r)$, $r\geqslant 1$, is a growth condition on the function $F$ around any minimizer (any critical point in the non-convex case). It is sometimes also called $r$-conditioning \cite{garrigos2017convergence} or H\"olderian error bounds \cite{Bolte2017}. This assumption is motivated by the fact that, when $F$ is convex, $\hd(r)$ is equivalent to the famous \L ojasiewicz inequality \cite{Loja63,Loja93}, a key tool in the mathematical analysis of continuous (or discrete) subgradient dynamical systems, with exponent $\theta = 1-\frac{1}{r}$:
\begin{definition}
\label{def_loja}
A differentiable function $F:\mathbb R^n \to \mathbb R$ is said to have the \L ojasiewicz property with exponent $\theta \in [0,1)$ if, for any critical point $x^*$, there exist $c> 0$ and $\varepsilon >0$ such that:
\begin{equation}
\forall x\in B(x^*,\varepsilon),~\|\nabla F(x)\| \geqslant c|F(x)-F(x^*)|^{\theta},\label{loja}
\end{equation}
where: $0^0=0$ when $\theta=0$ by convention.
\end{definition}

When the set $X^*$ of the minimizers is a connected compact set, the \L ojasiewicz inequality turns into a geometrical condition on $F$ around its set of minimizers $X^*$, usually referred to as H\"older metric subregularity \cite{kruger2015error}, and whose proof can be easily adapted from \cite[Lemma 1]{AttouchBolte2009}:
\begin{lemma}
Let $F:\R^n\rightarrow \R$ be a convex differentiable function satisfying the growth condition $\hd(r)$ for some $r\geqslant 1$. Assume that the set $X^*=\argmin F$ is compact. Then there exist $K>0$ and $\varepsilon >0$ such that for all $x\in \R^n$:
$$d(x,X^*) \leqslant \varepsilon\Rightarrow K d(x,X^*)^{r}\leqslant F(x)-F^*.$$\label{lem:H2}
\end{lemma}
Typical examples of functions having the \L ojasiewicz property are real-analytic functions, $C^1$ subanalytic functions or semi-algebraic functions \cite{Loja63,Loja93}. Strongly convex functions satisfy a global \L ojasiewicz property with exponent $\theta=\frac{1}{2}$ \cite{AttouchBolte2009}, or equivalently a global version of the hypothesis $\hd(2)$, namely:
$$\forall x\in \R^n,  F(x)-F^*\geqslant \frac{\mu}{2}\|x-x^*\|^2,$$
where $\mu>0$ denotes the parameter of strong convexity and $x^*$ the unique minimizer of $F$. By extension, uniformly convex functions of order $p\geqslant 2$ satisfy the global version of the hypothesis $\hd(p)$ \cite{garrigos2017convergence}.

Let us now present two simple examples of convex differentiable functions to illustrate situations where the hypothesis $\hu$ and $\hd$ are satisfied. Let $\gamma > 1$ and consider the function defined by: $F:x\in \R\mapsto |x|^\gamma$. 
We easily check that $F$ satisfies the hypothesis $\hu(\gamma')$ for some $\gamma'\geq 1$ if and only if $\gamma'\in [1,\gamma]$. By definition, $F$ also naturally satisfies $\hd(r)$ if and only if $r\geqslant \gamma$. Same conditions on $\gamma'$ and $r$ can be derived without uniqueness of the minimizer for functions of the form:
\begin{equation}
F(x) = \left\{\begin{array}{ll}
\max(|x|-a,0)^\gamma &\mbox{ if } |x| \geqslant a,\\
0 &\mbox{otherwise,}
\end{array}\right.\label{ex2}
\end{equation}
with $a>0$, and whose set of minimizers is: $X^*=[-a,a]$, since conditions $\hu(\gamma)$ and $\hd(r)$ only make sense around the extremal points of $X^*$.

Let us now investigate the relation between the parameters $\gamma$ and $r$ in the general case: any  convex differentiable function $F$ satisfying both $\hu(\gamma)$ and $\hd(r)$, has to be at least as flat as $x\mapsto \|x\|^\gamma$ and as sharp as $x\mapsto \|x\|^r$ in the neighborhood of its minimizers. Combining the flatness condition $\hu(\gamma)$ and the growth condition $\hd(r)$, we consistently deduce:  
\begin{lemma}
If a convex differentiable function satisfies both $\hu(\gamma)$ and $\hd(r)$ then necessarily $r\geqslant \gamma$. \label{lem:geometry2}
\end{lemma}

Finally, we conclude this section by showing that an additional assumption of the Lipschitz continuity of the gradient provides additional information on the local geometry of $F$: indeed, for convex functions, the Lipschitz continuity of the gradient is equivalent to a quadratic upper bound on $F$:
\begin{equation}\label{lipschitz}
\forall (x,y)\in \R^n\times \R^n, ~F(x)-F(y) \leqslant \langle \nabla F(y),x-y \rangle + \frac{L}{2}\|x-y\|^{2}.
\end{equation}
Applying \eqref{lipschitz} at $y=x^*$, we then deduce:
\begin{equation}\label{lipschitz2}
\forall x\in \R^n,~F(x)-F^* \leqslant \frac{L}{2}\|x-x^*\|^{2},
\end{equation}
which indicates that $F$ is at least as flat as $\|x-x^*\|^2$ around $X^*$. More precisely:

\begin{lemma}
Let $F:\R^n \rightarrow \R$ be a convex differentiable function 
with a $L$-Lipschitz continuous gradient for some $L>0$. 
Assume also that $F$ satisfies the growth condition $\hd(2)$ for some constant $K>0$. Then $F$ 
automatically satisfies $\hu(\gamma)$ with $\gamma=1+\frac{K}{2L}\in(1,2]$.\label{lem:Lipschitz}
\end{lemma}
\begin{proof}
Since $F$ is convex with a Lipschitz continuous gradient, we have:
$$\forall (x,y)\in \R^n, F(y)-F(x)-\langle \nabla F(x),y-x\rangle \geqslant \frac{1}{2L}\|\nabla F(y)-\nabla F(x)\|^2,$$
hence:
$$\forall x\in \R^n, F(x)-F^*\leqslant \langle\nabla F(x),x-x^*\rangle -\frac{1}{2L}\|\nabla F(x)\|^2.$$
Assume in addition that $F$ satisfies the growth condition $\hd(2)$ for some constant $K>0$. Then $F$ has the \L ojasiewicz property with exponent $\theta=\frac{1}{2}$ and constant $c=\sqrt{K}$. Thus:
$$\left(1+\frac{K}{2L}\right)(F(x)-F^*) \leqslant\langle\nabla F(x),x-x^*\rangle,$$
in the neighborhood of its minimizers, which means that $F$ satisfies $\hu(\gamma)$ with $\gamma=1+\frac{K}{2L}$.
\end{proof}

\begin{remark}
Observe that Lemma \ref{lem:Lipschitz} can be easily extended to the case of convex differentiable functions with a $\nu$-H\"older continuous gradient. 
Indeed, let $F$ be a convex differentiable functions with a $\nu$-H\"older continuous gradient for some $\nu\geqslant 1$. If $F$ also satisfies the growth condition $\hd(1+\nu)$ (for some constant $K>0$), then $F$ automatically satisfies $\hu(\gamma)$ with $\gamma =1 + \frac{\alpha K}{(1+\nu)L^\frac{1}{\nu}}$. This result is based on a notion of generalized co-coercivity for functions having a H\"older continuous gradient.
\end{remark}

\section{Related results}\label{sec_state}
In this section, we recall some classical state of the art results on the convergence properties of the trajectories of the ODE \eqref{ODE}.

Let us first recall that as soon as $\f>0$,  $F(x(t))$ converges to $F^*$ \cite{AujolDossal,attouch2017rate}, but a larger value of $\f$ is required to show the convergence of the trajectory $x(t)$. More precisely, if $F$ is convex and $\alpha >3$, or if $F$ satisfies $\hu(\gamma)$ hypothesis and $\f>1+\frac{2}{\gamma}$ then:
\begin{equation*}
F(x(t))-F^*=o\left(\frac{1}{t^2}\right),
\end{equation*}
and the trajectory $x(t)$ converges (weakly in an infinite dimensional space)  to a minimizer $x^*$ of $F$ \cite{su2016differential,AujolDossal,may2015asymptotic}. This last point generalizes what is known on convex functions: thanks to the additional hypothesis $\hu(\gamma)$, the optimal decay $\frac{1}{t^2}$ can be achieved for a damping parameter $\alpha$ smaller that $3$.

In the sub-critical case (namely when $\alpha <3$), it has been proven in \cite{attouch2017rate,AujolDossal} that if $F$ is convex, the convergence rate is then given by:
\begin{equation}
F(x(t))-F^*=O\left(\frac{1}{t^\frac{2\alpha}{3}}\right),
\end{equation}
but we can no longer prove the convergence of the trajectory $x(t)$.

The purpose in this paper is to prove that by exploiting the geometry of the function $F$, better rates of convergence can be achieved for the values $F(x(t))-F^*$.

Consider first the case when $F$ is convex and $\alpha \leqslant 1+\frac{2}{\gamma}$. A first contribution  in this paper is to provide convergence rates for the values when $F$ only satisfies $\hu(\gamma)$. Although we can no longer prove the convergence of the trajectory $x(t)$, we still have the following convergence rate for $F(x(t))-F^*$:
\begin{equation}
F(x(t))-F^*=O\left(\frac{1}{t^{\frac{2\gamma\f}{2+\gamma}}}\right),
\end{equation}
and this decay is optimal and achieved for $F(x)=\vert x\vert^{\gamma}$ for any $\gamma\geqslant  1$. These results have been first stated and proved in the unpublished report \cite{AujolDossal} by Aujol and Dossal in 2017 for convex differentiable functions satisfying $(F-F^*)^\frac{1}{\gamma}$ convex. Observe that this decay is still valid for $\gamma=1$ i.e. with the sole assumption of convexity as shown in \cite{attouch2017rate}, and that the constant hidden in the big $O$ is explicit and available also for $\gamma<1$, that is for non-convex functions (for example for functions whose square is convex).

Consider now the case when $\alpha > 1+\frac{2}{\gamma}$. In that case, with the sole assumption $\hu(\gamma)$ on $F$ for some $\gamma \geqslant 1$, it is not possible to get a bound on the decay rate like $O(\frac{1}{t^\delta})$ with $\delta>2$. Indeed as shown in \cite[Example 2.12]{attouch2018fast}, for any $\eta>2$ and for a large friction parameter $\f$, the solution $x$ of the ODE associated to $F(x)=|x|^{\eta}$ satisfies:
$$F(x(t))-F^*=Kt^{-\frac{2\eta}{\eta-2}},$$ 
and the power $\frac{2\eta}{\eta-2}$ can be chosen arbitrary close to $2$. More conditions are thus needed to obtain a decay faster than $O\left(\frac{1}{t^2}\right)$, which is the uniform rate that can be achieved for $\alpha \geqslant 3$ for convex functions.

Our main contribution is to show that a flatness condition $\hu$ associated to classical sharpness conditions such as the \L ojasiewicz property provides new and better decay rates on the values $F(x(t))-F^*$, and to prove the optimality of these rates in the sense that they are achieved for instance for the function $F(x)=|x|^\gamma$, $x\in \R$, $\gamma\geqslant 1$.



We will then confront our results to well-known results in the literature. In particular we will focus on the case when $F$ is strongly convex or has a strong minimizer \cite{cabot2009long}. In that case, Attouch Chbani, Peypouquet and Redont in \cite{attouch2018fast} following Su, Boyd and Candes \cite{su2016differential} proved that for any $\f>0$ we have:
$$F(x(t))-F^*=O\left(t^{-\frac{2\alpha}{3}}\right),$$
(see also \cite{attouch2017rate} for more general viscosity term in that setting). In Section~\ref{sec_contrib}, we will prove the optimality of the power $\frac{2\alpha}{3}$ in \cite{attouch2016fast}, and that if $F$ has additionally  a Lipschitz gradient then the decay rate of $F(x(t))-F^*$ is always strictly better than $O\left(t^{-\frac{2\alpha}{3}}\right)$.

Eventually several results about the convergence rate of the solutions of ODE associated to the classical gradient descent : 
\begin{equation}\label{EDOGrad}
\dot{x}(t)+\nabla F(x(t))=0,
\end{equation}
or the ODE associated to the heavy ball method 
\begin{equation}\label{EDO}
\ddot{x}+\alpha\dot x(t)+\nabla F(x(t))=0
\end{equation}
under geometrical conditions such that the \L ojasiewicz property have been proposed, see for example Polyak-Shcherbakov~\cite{polyak2017lyapunov}. The authors prove that if the function $F$ satisfies $\hd(2)$ and some other conditions, the decay of $F(x(t))-F^*$ is exponential for the solutions of both previous equations. These rates are the continuous counterparts of the exponential decay rate of the classical gradient descent algorithm and the heavy ball method algorithm for strongly convex functions.  

In the next section we will prove that this exponential rate is not true for solutions of \eqref{ODE} even for quadratic functions, and we will prove that from an optimization point of view, the classical Nesterov acceleration may be less efficient than the classical gradient descent.

\section{Contributions}\label{sec_contrib}
In this section, we state the optimal convergence rates that can be achieved when $F$ satisfies hypotheses such as  $\hu(\gamma)$ and/or $\hd(r)$.
The first result gives optimal control for functions whose geometry is sharp :
\begin{theorem}\label{Theo1}
Let $\gamma\geqslant  1$ and $\alpha >0$. If $F$ satisfies $\hu(\gamma)$ and if $\alpha\leqslant 1+\frac{2}{\gamma}$ then:  
\begin{equation*}
F(x(t))-F^*=O\left(\frac{1}{t^{\frac{2\gamma\f}{\gamma+2}}}\right).
\end{equation*}
\end{theorem}
\begin{figure}[h]
\includegraphics[width=\textwidth]{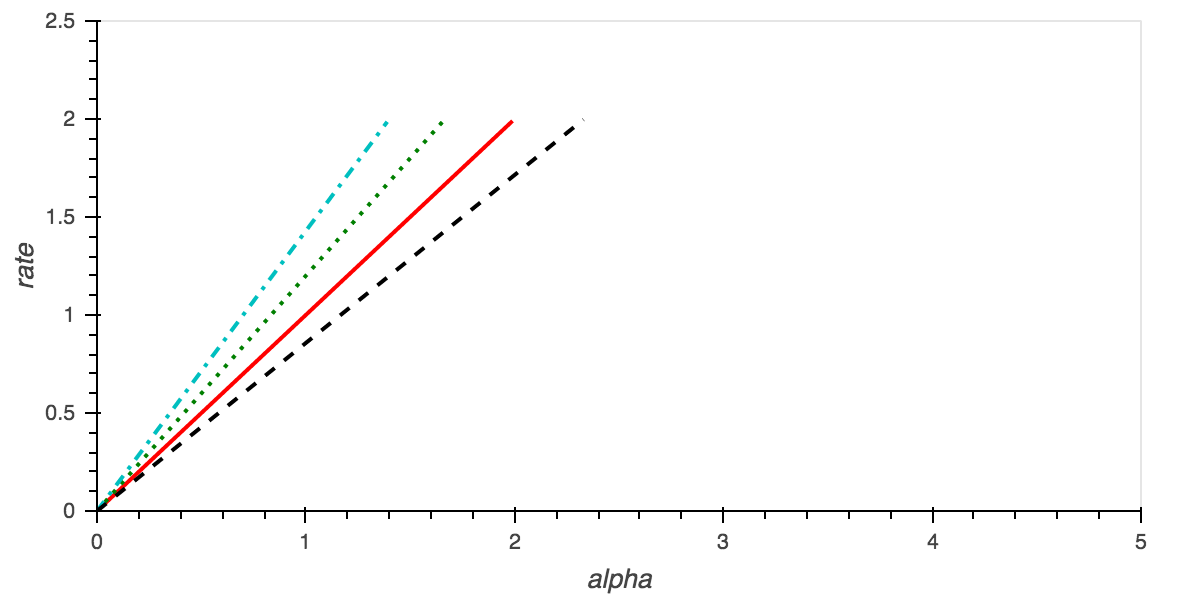}
\caption{Decay rate $r(\alpha,\gamma)=\frac{2\alpha\gamma}{\gamma+2}$ depending on $\alpha$ when $\alpha\leqslant 1+\frac{2}{\gamma}$ and when $F$ satisfies $\hu(\gamma)$ (as in Theorem \ref{Theo1}) for four values $\gamma$: 
$\gamma_1=1.5$ dashed line, $\gamma_2=2$,  solid line, $\gamma_3=3$ dotted line and $\gamma_4=5$ dashed-dotted line.}
\end{figure}

Note that a proof of the Theorem \ref{Theo1} has been proposed in the unpublished report \cite{AujolDossal}. The obtained decay is proved to be optimal in the sense that it can be achieved for some explicit functions $F$ for any $\gamma <1$. As a consequence one cannot expect a $o(t^{-\frac{2\gamma\f}{\gamma+2}})$ decay when $\alpha <1+\frac{2}{\gamma}$.

Let us now consider the case when $\alpha > 1+\frac{2}{\gamma}$. The second result in this paper provides optimal convergence rates for functions whose geometry is sharp, with a large friction coefficient:
\begin{theorem}\label{Theo1b}
Let $\gamma\geqslant  1$ and 
$\alpha >0$. If $F$ satisfies $\hu(\gamma)$ and $\hd(2)$ for some $\gamma\leqslant 2$, if $F$ has a unique minimizer and if $\alpha>1+\frac{2}{\gamma}$ then
\begin{equation*}\label{eqTheo1}
F(x(t))-F^*=O\left(\frac{1}{t^{\frac{2\gamma\f}{\gamma+2}}}\right).
\end{equation*} 
Moreover this decay is optimal in the sense that for any $\gamma\in(1,2]$ this rate is achieved for the function $F(x)=\vert x\vert^\gamma$.
\end{theorem}
\begin{figure}[h]
\includegraphics[width=\textwidth]{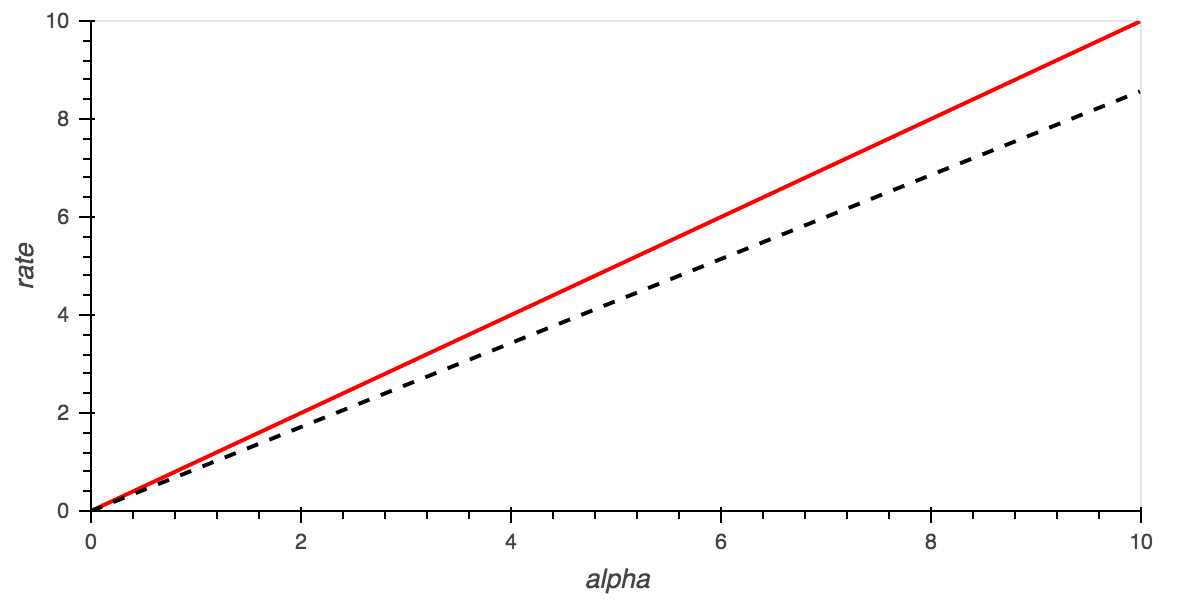}
\caption{Decay rate $r(\alpha,\gamma)=\frac{2\alpha\gamma}{\gamma+2}$ depending on the value of $\alpha$ when $F$ satisfies $\hu(\gamma)$ and $\hd(2)$ (as in Theorem \ref{Theo1b}) with 
$\gamma\leqslant 2$ for two values $\gamma$ : 
$\gamma_1=1.5$ dashed line, $\gamma_2=2$,  solid line.}
\end{figure}
Note that Theorem~\ref{Theo1b} only applies for $\gamma\leqslant 2$, since there is no function that satisfies both conditions $\hu(\gamma)$ with $\gamma>2$ and $\hd(2)$ (see Lemma \ref{lem:geometry2}). The optimality of the convergence rate result is precisely stated in the next Proposition:
\begin{proposition} \label{prop_optimal}
Let $\gamma\in (1,2]$.
Let us assume that $\alpha>0$. Let $x$ be a solution of \eqref{ODE} with $F(x)=\vert x\vert^{\gamma}$, $|x(t_0)|<1$ and $\dot{x}(t_0)=0$ where $t_0>\sqrt{\max(0,\frac{\alpha \gamma(\alpha -1-2/\gamma)}{(\gamma +2)^2})}$. There exists $K >0$ such that for any $T>0$, there exists $t \geqslant T$ such that
\begin{equation}
F(x(t))-F^* \geqslant \frac{K}{t^{\frac{2 \gamma \alpha}{\gamma +2}}}.
\end{equation}
\end{proposition}

Let us make several observations: first, to apply Theorem~\ref{Theo1b}, more conditions are needed than for Theorem~\ref{Theo1}: the hypothesis $\hd(2)$ and the uniqueness of the minimizer are needed to prove a decay faster than $O(\frac{1}{t^2})$, which is the uniform rate than can be achieved with $\alpha\geqslant  3$ for convex functions \cite{su2016differential}. The uniqueness of the minimizer is crucial in the proof of Theorem~\ref{Theo1b}, but it is still an open problem to know if this uniqueness is a necessary condition. In particular, observe that if $\dot x(t_0)=0$, then for all $t \geqslant t_0$, $x(t)$ belongs to $x_0+ {\rm Im} (\nabla F) $ 
where ${\rm Im} (\nabla F) $ stands for the vector space generated by $\nabla F (x)$ for all $x$ in $\R^n$.
As a consequence, Theorem~\ref{Theo1b} still holds true as long as the assumptions are valid in  $x_0+ {\rm Im} (\nabla F) $. 
\begin{remark}[The Least-Square problem]
Let us consider the
classical Least-Square problem defined by:
$$\displaystyle\min_{x\in \R^n}F(x):=\frac{1}{2}\|Ax-b\|^2,$$
where $A$ is a linear operator and $b\in \R^n$.
If $\dot x(t_0)=0$, then for all $t \geqslant t_0$, we have thus that
$x(t)$ belongs to the affine subspace $x_0+{\rm Im}(A^*)$. Since we have uniqueness of the solution on $x_0+{\rm Im}(A^*)$, Theorem~\ref{Theo1b} can be applied.
\end{remark}

We can also remark that if $F$ is a quadratic function in the neighborhood of $x^*$,  then $F$ satisfies $\hu(\gamma)$ for any $\gamma \in [1,2]$. Consequently, 
Theorem~\ref{Theo1b}  applies with $\gamma=2$ and thus:
\begin{equation*}
F(x(t))-F^*=O\left(\frac{1}{t^{\f}}\right).
\end{equation*} 
Observe that the optimality result provided by the Proposition~\ref{prop_optimal} ensures that we cannot expect an exponential decay of $F(x(t))-F^*$ for quadratic functions whereas this exponential decay can be achieved for the ODE associated to Gradient descent or Heavy ball method \cite{polyak2017lyapunov}.

Likewise, if $F$ is a convex differentiable function with a Lipschitz continuous gradient, and if $F$ satisfies the growth condition $\hd(2)$, then $F$ automatically satisfies the assumption $\hu(\gamma)$ with some $1<\gamma\leqslant 2$ as shown by Lemma~\ref{lem:Lipschitz}, and Theorem~\ref{Theo1b} applies with $\gamma>1$. 

Finally if $F$ is strongly convex or has a strong minimizer, then $F$ naturally satisfies $\hu(1)$ and a global version of $\hd(2)$. Since we prove the optimality  of the decay rates given by Theorem~\ref{Theo1b}, a consequence of this work is also the optimality of the power $\frac{2\alpha}{3}$ in \cite{attouch2016fast} for strongly convex functions and functions having a strong minimizer.

In both cases, we thus obtain convergence rates which are strictly better than $O(t^{-\frac{2\f}{3}})$ that is proposed for strongly convex functions by Su et al. \cite{su2016differential} and Attouch et al. \cite{attouch2018fast}. Finally it is worth noticing that the decay for strongly convex functions is not exponential while it is the case for the classical gradient descent scheme (see e.g. \cite{garrigos2017convergence}). This shows that applying the classical Nesterov acceleration on convex functions without looking more at the geometrical properties of the objective functions may lead to sub-optimal algorithms.

Let us now focus on flat geometries i.e. geometries associated to $\gamma>2$. Note that the uniqueness of the minimizer is not need anymore:
\begin{theorem}\label{Theo2}
Let $\gamma_1>2$ and $\gamma_2 >2$. Assume that $F$ is coercive and satisfies $\hu(\gamma_1)$ and $\hd(\gamma_2)$ with $\gamma_1\leqslant \gamma_2$. If $\f\geqslant  \frac{\gamma_1+2}{\gamma_1-2}$  then we have: 
\begin{equation*}\label{eqTheo2}
F(x(t))-F^*=O\left(\frac{1}{t^{\frac{2\gamma_2}{\gamma_2-2}}}\right).
\end{equation*} 
\end{theorem}
In the case when $\gamma_1= \gamma_2$, we have furthermore the convergence of the trajectory:
\begin{corollary}\label{Corol2}
Let $\gamma>2$. If $F$ is coercive and satisfies $\hu(\gamma)$ and $\hd(\gamma)$, and if 
$\f\geqslant  \frac{\gamma+2}{\gamma-2}$ then we have:
\begin{equation*}\label{eqCorol2}
F(x(t))-F^*=O\left(\frac{1}{t^{\frac{2\gamma}{\gamma-2}}}\right),
\end{equation*} 
and
\begin{equation}
\norm{\dot x(t)}=O\left(\frac{1}{t^{\frac{\gamma}{\gamma-2}}}\right).
\end{equation}
Moreover the trajectory $x(t)$ has a finite length and it converges to a minimizer $x^*$ of $F$.
\end{corollary}
\begin{figure}[h]
\includegraphics[width=\textwidth]{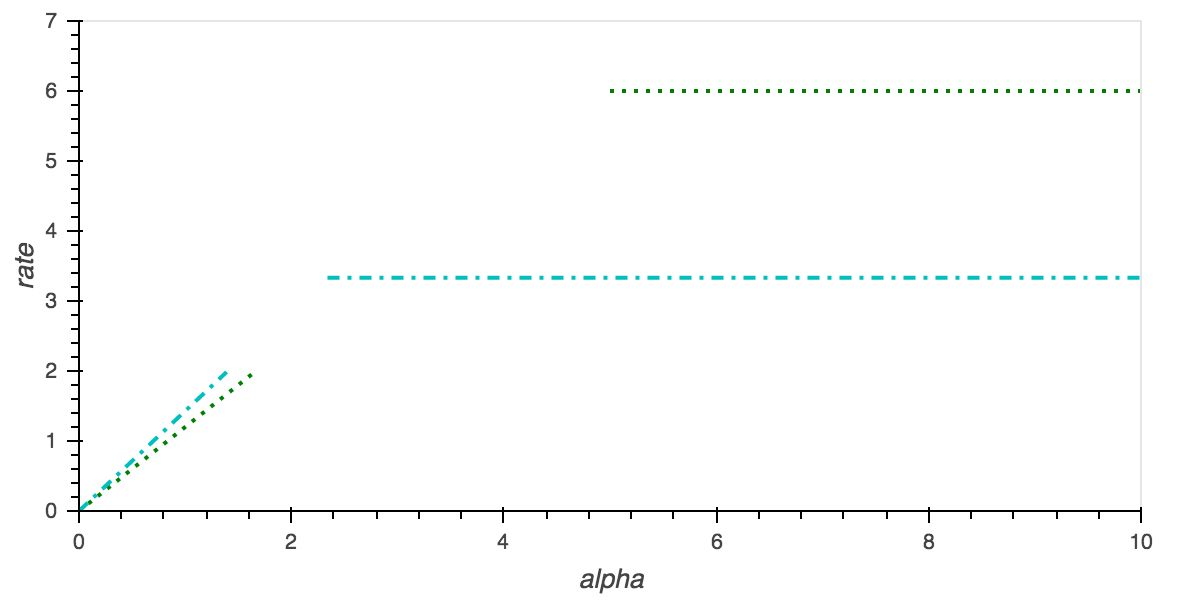}
\caption{Decay rate $r(\alpha,\gamma)=\frac{2\gamma}{\gamma-2}$ depending on the value of $\alpha$ when $\alpha\geqslant \frac{\gamma+2}{\gamma-2}$ 
when $F$ satisfies $\hu(\gamma)$ (as in Theorem \ref{Theo2}) for two values $\gamma$: $\gamma_3=3$ dotted line and $\gamma_4=5$ dashed-dotted line.}\label{fig:flat1}
\end{figure}
\begin{figure}[h]
\includegraphics[width=\textwidth]{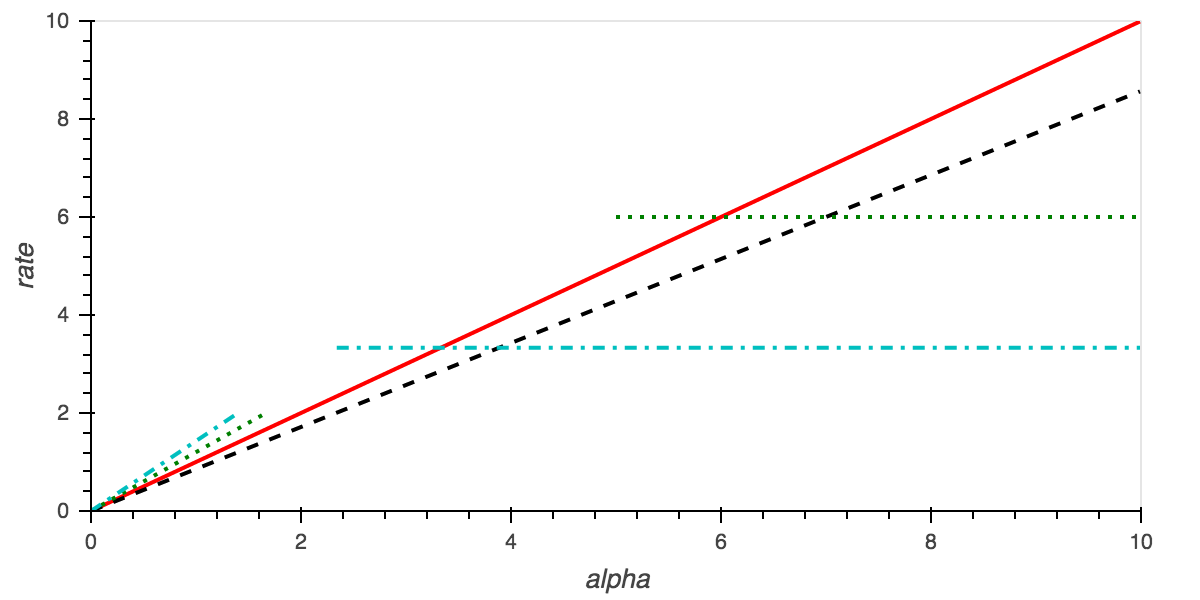}
\caption{Decay rate $r(\alpha,\gamma)$ depending on the value of $\alpha$ if $F$  satisfies $\hu(\gamma)$ and $\hd(r)$ with $r=\max(2,\gamma)$ for four values $\gamma$ : 
$\gamma_1=1.5$ dashed line, $\gamma_2=2$,  solid line, $\gamma_3=3$ dotted line and $\gamma_4=5$ dashed-dotted line.}\label{fig:flat2}
\end{figure}
Observe that the decay obtained in Corollary \ref{Corol2} is optimal since Attouch et al. proved that it is achieved for the function $F(x)=\vert x\vert^\gamma$ in \cite{attouch2018fast}.\\

From Theorems \ref{Theo1}, \ref{Theo1b} and \ref{Theo2}, we can make the following comments: first in Theorems~\ref{Theo1b} and \ref{Theo2}, both conditions $\hu$ and $\hd$ are used to get a decay rate and it turns out that these two conditions are important. 

With the sole hypothesis $\hd(\gamma)$ it seems difficult to establish optimal rate. Consider for instance the function $F(x)=|x|^3$ which satisfies $\hu(3)$ and $\hd(3)$. Applying Theorem \ref{Theo2} with $\gamma_1=\gamma_2=3$, we know that for this function with $\alpha=\frac{\gamma_1+2}{\gamma_1-2}=5$, we have $F(x(t))-F^*=O\left(\frac{1}{t^6}\right)$. But, with the sole hypothesis $\hd(3)$, such a decay cannot be achieved. Indeed,
the function $F(x)=|x|^{2}$ satisfies $\hd(3)$, but from the optimality part of Theorem \ref{Theo1b} we know that we cannot achieve a decay better than $\frac{1}{t^{\frac{2\alpha \gamma}{\gamma+2}}}=\frac{1}{t^5}$ for $\f=5$. 

Consider now a convex function $F$ behaving like $\norm{x-x^*}^{\gamma}$ in the neighborhood of its unique minimizer $x^*$. The decay of $F(x(t))-F^*$ then depends directly on $\alpha$ if $\gamma\leqslant 2$, but it does not depend on $\alpha$ for large $\alpha$ if $\gamma>2$. Moreover for such functions the best decay rate of $F(x(t))-F^*$ is $O\left(\frac{1}{t^{\alpha}}\right)$ and is achieved for $\gamma=2$ i.e. for quadratic like functions around the minimizer. If $\gamma<2$, it seems that the oscillations of the solution $x(t)$ prevent us from getting an optimal decay rate. 
The inertia seems to be too large for such functions. If $\gamma>2$, for large $\alpha$, the decay is not as fast because the gradient of the functions decays too fast in the neighborhood of the minimizer. For these functions a larger inertia could be more efficient.

Finally, observe that as shown in Figures \ref{fig:flat1} and \ref{fig:flat2}, the case when $1+\frac{2}{\gamma}<\alpha<\frac{\gamma+2}{\gamma-2}$ is not covered by our results. Although we did not get a better convergence rate than $\frac{1}{t^2}$ in that case, we can prove that there exist some initial conditions for which the convergence rate can not be better than $t^{-\frac{2\gamma\alpha}{\gamma+2}}$:
\begin{proposition}\label{PropOpt2}
Let $\gamma >2$ and $1+\frac{2}{\gamma}<\alpha<\frac{\gamma+2}{\gamma-2}$. Let $x$ be a solution of \eqref{EDO} with $F(x)=|x|^\gamma$, $|x(t_0)|<1$ and $\dot x(t_0)=0$ for any given $t_0>0$. Then there exists $K>0$ such that for any $T>0$, there exists $t\geqslant T$ such that:
$$F(x(t))-F^* \geqslant \dfrac{K}{t^{\frac{2\gamma\alpha}{\gamma+2}}}.$$\label{prop:gap}
\end{proposition}

\paragraph{Numerical Experiments}
In the following numerical experiments, the optimality of the decays given in all previous theorems, are tested for various choices of $\alpha$ and $\gamma$. 

More precisely we use a discrete Nesterov scheme to approximate the solution of \eqref{ODE} for $F(x)=|x|^{\gamma}$ on the interval $[t_0,T]$ with $t_0=0$ and $\dot{x}(t_0)=0$, see \cite{su2016differential}.

If $\gamma\geqslant 2$, $\nabla F$ is a Lipschitz function and we define the sequence $(x_n)_{n\in\mathbb{N}}$ as follows:
\begin{equation*} 
x_n=y_n-h\nabla F(y_n)\text{ with }y_n=x_n+\frac{n}{n+\alpha}(x_n-x_{n-1}),
\end{equation*}
where $h\in(0,1)$ is a time step. 

If $\gamma<2$, we use a proximal step : 
\begin{equation*} 
x_n=prox_{h F}(y_n)\text{ with }y_n=x_n+\frac{n}{n+\alpha}(x_n-x_{n-1}).
\end{equation*}
It has been shown that $x_n\approx x(n\sqrt{h})$ where the function $x$ is a solution of the ODE \eqref{ODE}.
 
In the following numerical experiments the sequence $(x_n)_{n\in\mathbb{N}}$ is computed for various pairs $(\gamma,\alpha)$. The step size is always set to $h=10^{-7}$. 

We define the function $rate(\alpha,\gamma)$ as the expected rate given in all the previous theorems and Proposition \ref{PropOpt2}, that is:
\begin{eqnarray*}
rate(\alpha,\gamma)&:=&\left\{\begin{array}{ll}
\dfrac{2\alpha \gamma}{\gamma+2} &\text{ if } \gamma\leqslant 2 \text{ or if }\gamma>2\text{ and }\alpha\leqslant 1+\frac{2}{\gamma}, \\
\dfrac{2\gamma}{\gamma-2}& \text{ if } \gamma>2\text{ and }\alpha\geqslant \frac{\gamma+2}{\gamma-2}, \\
\dfrac{2\alpha \gamma}{\gamma+2} &\text{ if } \gamma>2 \text{ and } \alpha\in(1+\frac{2}{\gamma}, \frac{\gamma+2}{\gamma-2}).
\end{array}\right.
\end{eqnarray*}

If the function $z(t):=\left(F(x(t))-F(x^*)\right)t^{\delta}$ is bounded but does not tend to 0, we can deduce that 
$\delta$ is the largest value such that 
$F(x(t))-F(x^*)=O\left(t^{-\delta}\right)$. 
We define
\begin{equation*}
z_n:=(F(x_n)-F(x^*))\times (n\sqrt{h})^{rate(\alpha,\gamma)}\approx (F(x(t))-F(x^*))t^{rate(\alpha,\gamma)},
\end{equation*} 
and if the function $rate(\alpha,\gamma)$ is optimal we expect that the sequence $(z_n)_{n\in\mathbb{N}}$ is bounded but do not decay to 0.  
The following figures give for various choices of $(\alpha,\gamma)$ the trajectory of the sequence $(z_n)_{n\in\mathbb{N}}$. The values are re-scaled such that the maximum is always $1$. In all these numerical examples, we will observe that the sequence $(z_n)_{n\in\mathbb{N}}$ is bounded and does not tend to $0$.

\begin{figure}[h]
\includegraphics[width=0.495\textwidth]{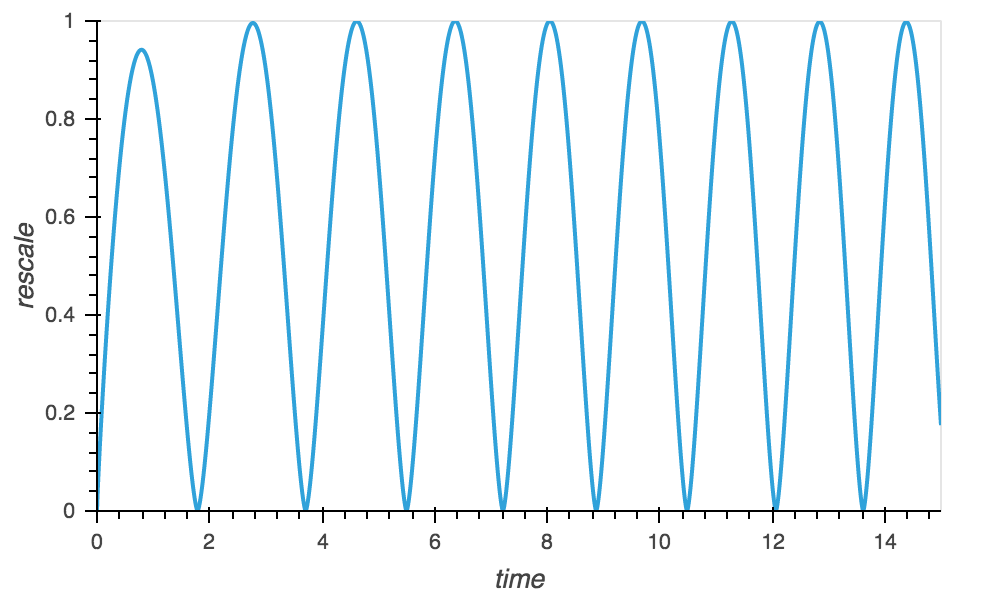}
\includegraphics[width=0.495\textwidth]{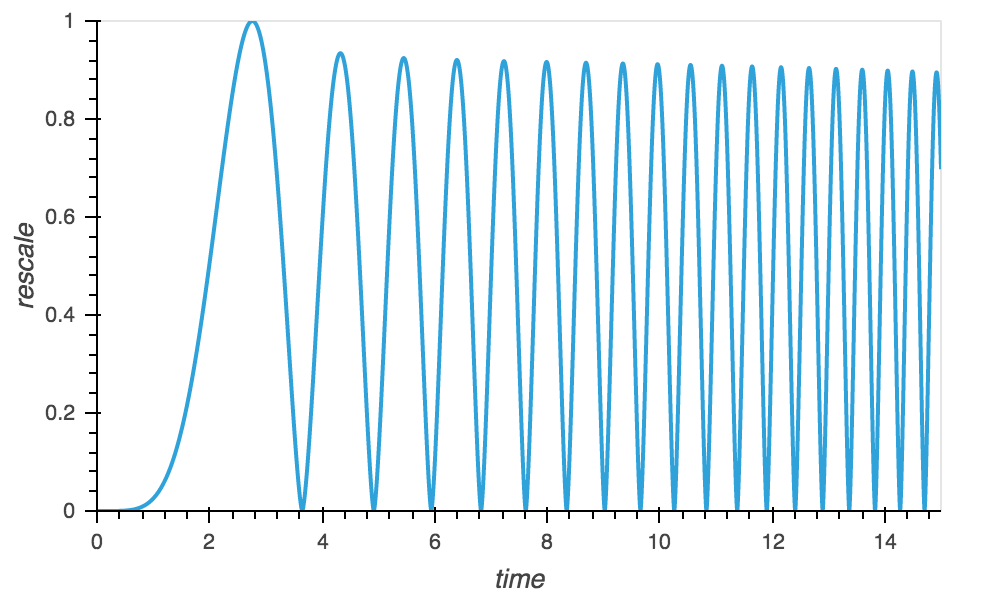}
\caption{Case when $\gamma=1.5$. On the left $\alpha=1$ and $rate(\alpha,\gamma)=\frac{2\alpha\gamma}{\gamma+2}=\frac{6}{7}$.
On the right $\alpha=6$ and $rate(\alpha,\gamma)=\frac{2\alpha\gamma}{\gamma+2}=\frac{36}{7}$}\label{fig:gamma1.5}
\end{figure}
\begin{figure}[h]
\includegraphics[width=0.495\textwidth]{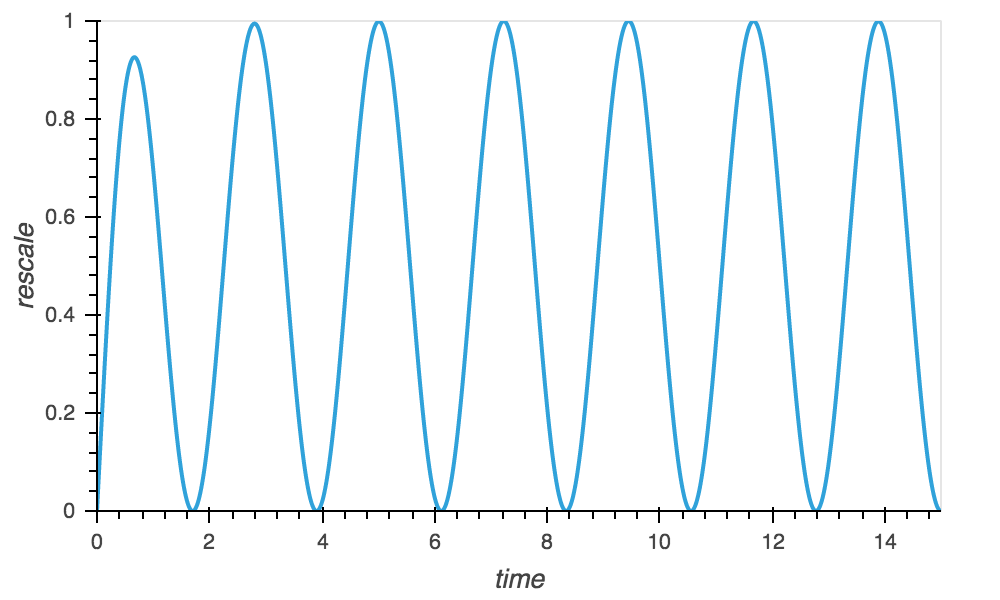}
\includegraphics[width=0.495\textwidth]{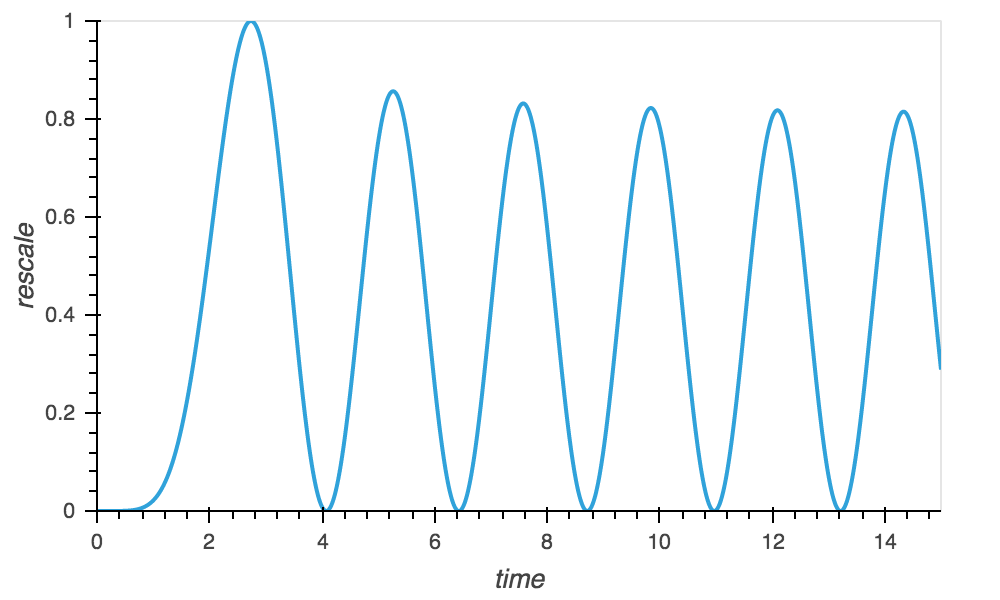}
\caption{Case when $\gamma=2$. On the left $\alpha=1$ and $rate(\alpha,\gamma)=\frac{2\alpha\gamma}{\gamma+2}=1$.
On the right $\alpha=6$ and $rate(\alpha,\gamma)=\frac{2\alpha\gamma}{\gamma+2}=6$}\label{fig:gamma2}
\end{figure}
\begin{figure}[h]
\includegraphics[width=0.495\textwidth]{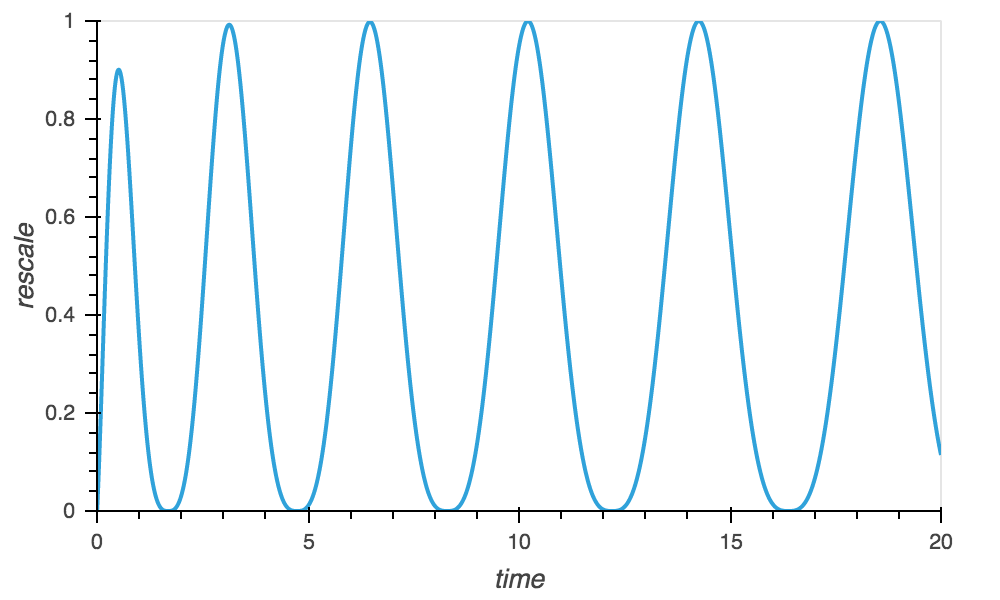}
\includegraphics[width=0.495\textwidth]{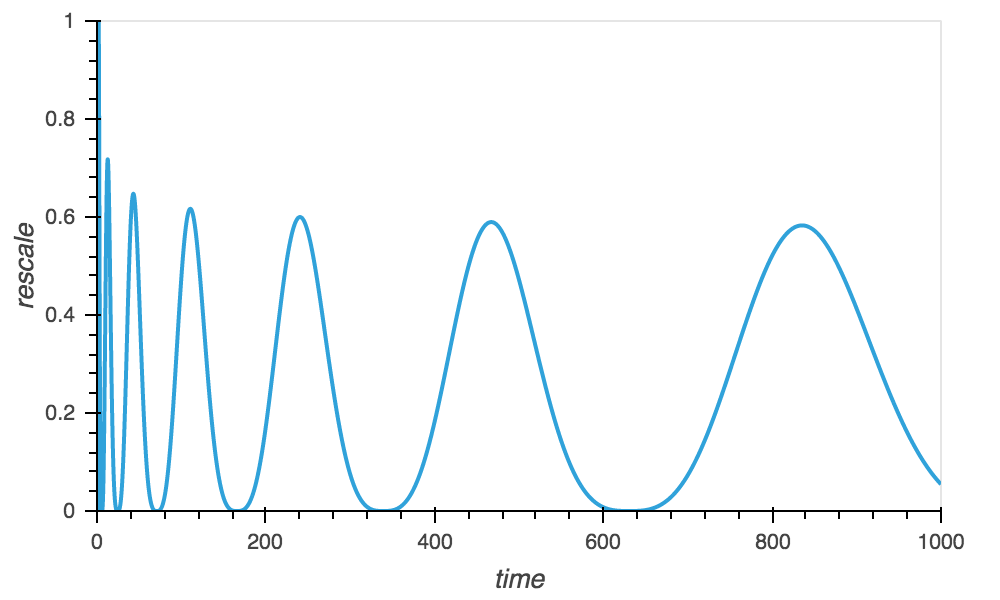}
\includegraphics[width=0.495\textwidth]{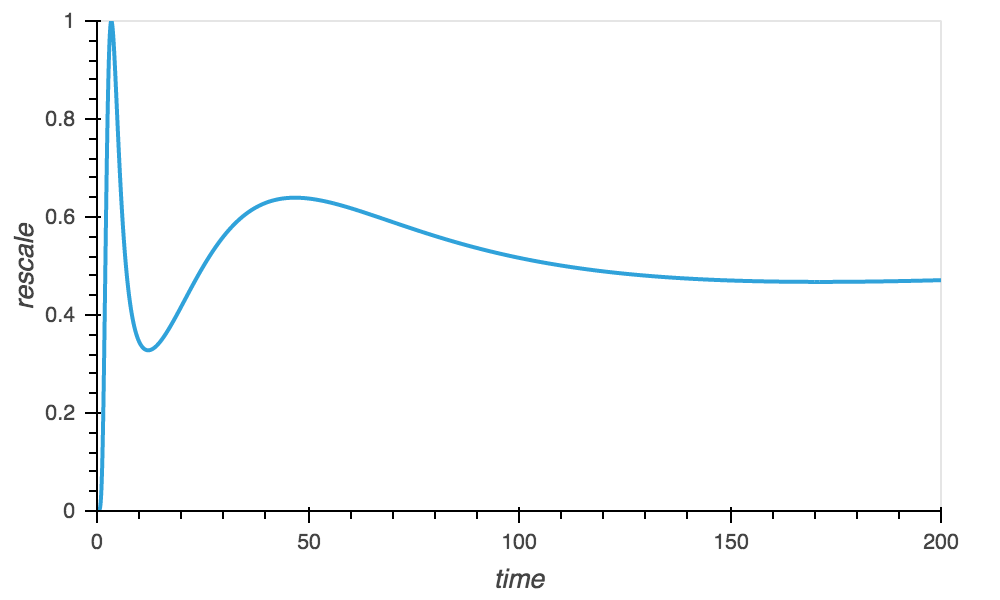}
\includegraphics[width=0.495\textwidth]{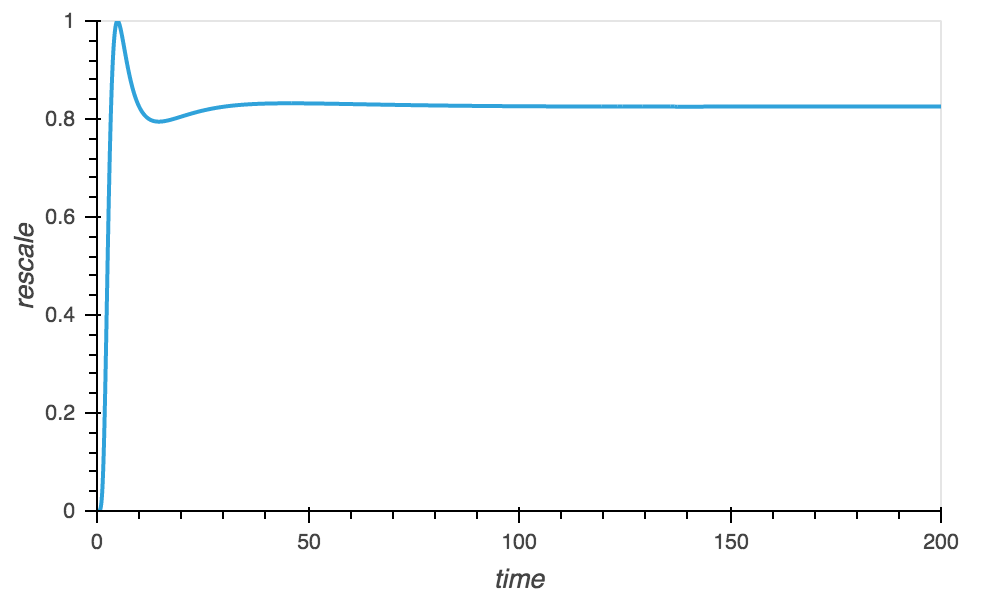}
\caption{Case when $\gamma=3$. On the top left $\alpha=1$ and $rate(\alpha,\gamma)=\frac{2\alpha\gamma}{\gamma+2}=1.2$,
on the top right $\alpha=4$ and $rate(\alpha,\gamma)=\frac{2\alpha\gamma}{\gamma+2}=4.8$,
on bottom left $\alpha=6$ and $rate(\alpha,\gamma)=\frac{2\gamma}{\gamma-2}=6$,
on bottom right $\alpha=8$ and $rate(\alpha,\gamma)=\frac{2\gamma}{\gamma-2}=6$}\label{fig:gamma3}
\end{figure}

\begin{itemize}
\item The Figures \ref{fig:gamma1.5} and \ref{fig:gamma2} with $\gamma=1.5$ and $\gamma=2$ illustrate Theorem \ref{Theo1}, Theorem \ref{Theo1b} and Proposition \ref{prop_optimal}. Indeed for sharp functions (i.e for $\gamma\leqslant 2$) the rate is proved to be optimal.
\item In the case $\gamma=3$ and $\alpha=1$, the fact that $(F(x(t))-F(x^*))t^{rate(\alpha,\gamma)}$ is bounded is also a consequence of Theorem \ref{Theo1}.  The optimality of this rate is not proven but the experiments show that it numerically is. 
\item In the case $\gamma=3$ and $\alpha=4$,  $\alpha\in(\frac{\gamma+2}{\gamma},\frac{\gamma+2}{\gamma-2})$ then 
the fact that $(F(x(t))-F(x^*))t^{rate(\alpha,\gamma)}$ is bounded is not proved but the experiments from Figure \ref{fig:gamma3} show that it numerically is. However Proposition \ref{PropOpt2} 
 proves that the sequence $(z_n)_{n\in\mathbb{N}}$ does not tend to 0, which is illustrated by the experiments.
\item When $\gamma=3$ and $\alpha=6$ or $\alpha=8$, Theorem \ref{Theo2} ensures that the sequence $(z_n)_{n\in\mathbb{N}}$ is bounded. This rate is proved to be optimal and the numerical experiments from Figure \ref{fig:gamma3} show that this rate is actually achieved for this specific choice of parameters.  
\end{itemize}

\section{Proofs}\label{sec_proofs}

In this section, we detail the proofs of the results presented in Section~\ref{sec_contrib}, namely Theorems \ref{Theo1}, \ref{Theo1b} and \ref{Theo2}, Propositions~\ref{prop_optimal} and \ref{prop:gap},  Corollary~\ref{Corol2}.

The proofs of the theorems rely on Lyapunov functions $\EE$ and $\HH$ introduced by Su, Boyd and Candes \cite{su2016differential}, Attouch, Chbani, Peypouquet and Redont \cite{attouch2018fast} and Aujol-Dossal \cite{AujolDossal} :  
\begin{equation*}
\mathcal{E}(t)=t^2(F(x(t))-F^*)+\frac{1}{2}
\norm{\lambda(x(t)-x^*)+t\dot{x}(t)}^2+\frac{\xi}{2}\norm{x(t)-x^*}^2,
\end{equation*}
where $x^*$ is a minimizer of $F$ and $\lambda$ and $\xi$ are two real numbers. 
The function $\HH$ is defined from $\EE$ and it depends on another real parameter $p$ : 
\begin{equation*}
\HH(t)=t^p\EE(t).
\end{equation*} 
Using the following notations:
\begin{align*}
a(t)&=t(F(x(t))-F^*),\\
b(t)&=\frac{1}{2t}\norm{\lambda(x(t)-x^*)+t\dot{x}(t)}^2,\\
c(t)&=\frac{1}{2t}\norm{x(t)-x^*}^2,
\end{align*} 
we have:
\begin{equation*}
\EE(t)=t(a(t)+b(t)+\xi c(t)).
\end{equation*}
From now on we will choose 
\begin{equation*}
\xi=\lambda(\lambda+1-\f),
\end{equation*}
and we will use the following Lemma whose proof is postponed to Appendix~\ref{appendix}:
\begin{lemma}\label{LemmeFonda}
If $F$ satisfies $\hu(\gamma)$ for any $\gamma\geq 1$, and if $\xi=\lambda(\lambda-\f+1)$ then 
 \begin{equation*}
 \HH'(t)\leqslant t^{p}\left((2-\gamma\lambda+p)a(t)+(2\lambda+2-2\f+p)b(t)+\lambda(\lambda+1-\f)(-2\lambda+p)c(t)\right).
 \end{equation*} 
\end{lemma} 
Note that this inequality is actually an equality for the specific choice $F(x)=\vert x\vert^\gamma$, $\gamma>1$.
\subsection{Proof of Theorems \ref{Theo1} and \ref{Theo1b}}
In this section we prove Theorem \ref{Theo1} and Theorem \ref{Theo1b}. Note that a complete proof of Theorem~\ref{Theo1}, including the optimality of the rate, can be found in the unpublished report \cite{AujolDossal} under the hypothesis that $(F-F^*)^\frac{1}{\gamma}$ is convex. The proof of both Theorems are actually similar. The choice of $p$ and $\lambda$ are the same but, to prove the first point, due to the value of $\alpha$, the function $\HH$ is non-increasing and sum of non-negative terms, which simplifies the analysis and necessitates less hypotheses to conclude.

We choose here $p=\frac{2\gamma \f}{\gamma+2}-2$ and $\lambda=\frac{2\f}{\gamma+2}$
and thus 
\begin{equation*}
\xi=\frac{2\alpha\gamma}{(\gamma+2)^2}(1+\frac{2}{\gamma}-\alpha).
\end{equation*}
From Lemma \ref{LemmeFonda}, it appears that:
\begin{equation}\label{ineqH1}
\HH'(t)\leqslant K_1t^{p}c(t)
\end{equation}
where the real constant $K_1$ is given by:
\begin{eqnarray*}
K_1 & = & \lambda(\lambda+1-\f)(-2\lambda+p)\\
&=& \frac{2\f}{\gamma+2} \left(\frac{2\f}{\gamma+2}+1-\f\right)
\left(-2 \frac{2\f}{\gamma+2}+\frac{2\gamma \f}{\gamma+2}-2
\right)
\\ & = &
\frac{4\f}{(\gamma+2)^3}
\left(2\f+\gamma+2-\f \gamma -2 \f \right)
\left(-2\f+\gamma \f - \gamma -2
\right)
\\ & = &
\frac{4\f}{(\gamma+2)^3}
\left(\gamma+2-\f \gamma\right)
\left(\f(-2+\gamma) - \gamma -2
\right).
\end{eqnarray*}
Hence:
\begin{equation}\label{eqdefK1}
K_1=\frac{4\f\gamma}{(\gamma+2)^3}
\left(1+\frac{2}{\gamma}-\f\right)
\left(\f(-2+\gamma) - \gamma -2
\right).
\end{equation}

Consider first the case when: $\f\leqslant 1+\frac{2}{\gamma}$. In that case, we observe that: $\xi\geq 0$, so that the energy $\HH$ is actually a sum of non-negative terms. Coming back to \eqref{ineqH1}, we have:
\begin{equation}
    \HH'(t)\leqslant K_1t^{p}c(t).\label{ineqH1b1}
\end{equation}
Since $\f\leqslant 1+\frac{2}{\gamma}$, the sign of the constant $K_1$ is the same as that of $\f(-2+\gamma) - \gamma -2$, and thus $K_1\leqslant 0$ for any $\gamma\geqslant 1$. According to \eqref{ineqH1b1}, the energy $\HH$ is thus non-increasing and bounded i.e.:
$$\forall t\geqslant t_0,~\HH(t)\leqslant \HH(t_0).$$
Since $\HH$ is a sum of non-negative terms, it follows directly that:
$$\forall t\geqslant t_0,~t^{p+2}(F(x(t))-F^*)\leqslant \HH(t_0),$$
which concludes the proof of Theorem~\ref{Theo1}.

Consider now the case when: $\alpha > 1+\frac{2}{\gamma}$. In that case, we first observe that: $\xi<0$, so that $\HH$ is not a sum of non-negative functions anymore, and an additional growth condition $\hd(2)$ will be needed to bound the term in $\norm{x(t)-x^*}^2$. Coming back to \eqref{ineqH1}, we have:
\begin{equation}
    \HH'(t)\leqslant K_1t^{p}c(t).\label{ineqH1b2}
\end{equation}
Since $\alpha > 1+\frac{2}{\gamma}$, the sign of the constant $K_1$ is the opposite of the sign of $\f(\gamma -2)-(\gamma+2)$. Moreover, since $\gamma\leqslant 2$, then $\f(\gamma -2)-(\gamma+2)<0$ and thus $K_1 >0$.

Using Hypothesis $\hd(2)$ and the uniqueness of the minimizer, there exists $K>0$ such that:
\begin{equation*}
Kt\norm{x(t)-x^*}^2\leqslant t(F(x(t))-F^*)=a(t),
\end{equation*}
and thus 
\begin{equation}
c(t)\leqslant \frac{1}{2Kt^2}a(t).\label{eqct}
\end{equation}
Since $\xi<0$ with our choice of parameters, we get: 
\begin{eqnarray}
\HH(t) &\geqslant & t^{p+1}(a(t) + \xi c(t))
\geqslant  t^{p+1}(1+\frac{\xi}{2Kt^2})a(t).\label{H:bound}
\end{eqnarray}
It follows that there exists $t_1$ such that for all $t\geqslant  t_1$, $\HH(t)\geqslant  0$ and: 
\begin{equation}
\HH(t) \geqslant \frac{1}{2}t^{p+1}a(t).\label{eqat}
\end{equation}
From \eqref{ineqH1b2}, \eqref{eqct} and \eqref{eqat}, we get:
\begin{equation*}
\HH'(t)\leqslant \frac{K_1}{K}\frac{\HH(t)}{t^3}.
\end{equation*}
From the Gr\"onwall Lemma in its differential form, there exists $A>0$ such that for all $t\geqslant  t_1$, we have: $\HH(t)\leqslant A$. According to \eqref{eqat}, 
we then conclude that $t^{p+2}(F(x(t))-F^*)=t^{p+1}a(t)$ is bounded which concludes the proof of Theorem \ref{Theo1b}.


\subsection{Proof of Proposition~\ref{prop_optimal} (Optimality of the convergence rates)}
Before proving the optimality of the convergence rate stated in Proposition~\ref{prop_optimal}, we need the following technical lemma:
\begin{lemma} \label{lemmatech}
Let $y$ a continuously differentiable function with values in $\mathbb R$.
Let $T>0$ and $\epsilon > 0$. If $y$ is bounded, then there exists $t_1 >T$ such that:
\begin{equation*} 
|\dot{y}(t_1) | \leqslant \frac{\epsilon}{t_1}.
\end{equation*}
\end{lemma}

\begin{proof}

We split the proof into two cases.
\begin{enumerate}
\item 
There exists $t_1 >T$ such that $\dot{y}(t_1)=0$.
\item
$\dot{y}(t)$ is of constant sign for $t> T$. For instance we assume $\dot{y}(t)>0$.
By contradiction, let us assume that $\dot{y}(t)>\frac{\epsilon}{t}$ $\forall t > T$.
Then $y(t)$ cannot be a bounded function as assumed.
\end{enumerate}
\end{proof}

Let us now prove the Proposition~\ref{prop_optimal}: the idea of the proof is the following: we first show that $\HH$ is bounded from below. Since $\HH$ is a sum of 3 terms including the term $F-F^*$, we then show that given $t_1 \geq t_0$, there always exists a time $t \geq t_1$ such that the value of $\HH$ is concentrated on the term $F-F^*$.

We start the proof by using the fact that, for the function $F(x)=\vert x\vert^{\gamma}$, $\gamma>1$, the inequality of Lemma~\ref{LemmeFonda} is actually an equality. Using the values $p=\frac{2\gamma \f}{\gamma+2}-2$ and $\lambda=\frac{2\f}{\gamma+2}$ of Theorems~\ref{Theo1} and \ref{Theo1b}, we have a closed form for the derivative of function $\HH$:
\begin{equation}\label{eqH1}
\HH'(t) = K_1 t^{p}c(t) =\frac{K_1}{2} t^{p-1}\vert x(t)\vert^{2},
\end{equation}
where $K_1$ is the constant given in \eqref{eqdefK1}. 
We will now prove that it exists $\ell>0$ such that for $t$ large enough:
\begin{equation*}
\HH(t) \geqslant \ell.
\end{equation*}
To prove that point we consider two cases depending on the sign of 
$\alpha-(1+\frac{2}{\gamma})$.
\begin{enumerate}
\item Case when $\alpha\leqslant 1+\frac{2}{\gamma}$, $\xi\geqslant 0$ and $K_1\leqslant 0$. We can first observe that $\HH$ is a non negative and non increasing function. 
Moreover it exists $\tilde t\geqslant t_0$ such that for $t\geqslant \tilde t$, $|x(t)|\leqslant 1$ and:
\begin{equation*}
t^pc(t)\leqslant \frac{t^pa(t)}{2t^2}\leqslant \frac{\HH(t)}{t^3},
\end{equation*}
which implies using \eqref{eqH1} that:
\begin{equation*}
|\HH'(t)|\leqslant |K_1|\frac{\HH(t)}{t^3}.
\end{equation*}
If we denote $G(t)=\ln (\HH(t))$ we get for all $t\geqslant \tilde t$,
\begin{equation*}
|G(t)-G(\tilde t)|\leqslant \int_{\tilde t}^t\frac{|K_1|}{s^3}ds.
\end{equation*}
We deduce that $|G(t)|$ is bounded below and then that it exists $\ell>0$ such that for $t$ large enough:
\begin{equation*}
\HH(t) \geqslant \ell,
\end{equation*}
\item Case when $\alpha> 1+\frac{2}{\gamma}$, $\xi< 0$ and $K_1> 0$. This implies in particular that $\HH$ is non-decreasing. Moreover, from Theorem~\ref{Theo1b}, $\HH$ is bounded above. Coming back to the inequality \eqref{H:bound}, we observe that $\HH(t_0)>0$ provided that $1+\frac{\xi}{2t_0^2}>0$, with $K=1$ and $\xi = \lambda(\lambda-\alpha+1)$, i.e.:
$$t_0>\sqrt{\frac{\alpha\gamma}{(\gamma +2)^2}(\alpha-(1+\frac{2}{\gamma}))}.$$

In particular, we have that for any $t\geqslant t_0$ 
\begin{equation*}
\HH(t) \geqslant \ell,
\end{equation*}
with $\ell=\HH(t_0)$
\end{enumerate}
Hence for any $\alpha>0$ and for $t$ large enough
\begin{equation*}
a(t)+b(t)+ \xi c(t) \geqslant \frac{\ell}{t^{p+1}}.
\end{equation*}
Moreover, since $c(t)=o(a(t))$ when $t \to + \infty$, we have that for $t$ large enough, 
\begin{equation*}
a(t)+b(t) \geqslant \frac{\ell}{2 t^{p+1}}.
\end{equation*}
%

Let $T>0$ and $\epsilon >0$. We set:
\begin{equation*}
y(t):=t^{\lambda} x(t),
\end{equation*}
where: $\lambda =\frac{2 \alpha}{\gamma+2}$. From the Theorem~\ref{Theo1} and Theorem \ref{Theo1b}, we know that $y(t)$ is bounded.
Hence, from Lemma~\ref{lemmatech}, there exists $t_1 > T$ such that 
\begin{equation} \label{eqt1}
|\dot{y}(t_1) |\leqslant \frac{\epsilon}{t_1}.
\end{equation}
But:
\begin{equation*}
\dot{y}(t) = t^{\lambda-1} \left(\lambda x(t)
+t \dot{x}(t)
\right).
\end{equation*}
Hence using \eqref{eqt1}:
\begin{equation*}
t_1^{\lambda} \left|\lambda x(t_1)
+t_1 \dot{x}(t_1)
\right|
\leqslant \epsilon.
\end{equation*}
We recall that: $b(t)=\frac{1}{2t}\norm{\lambda(x(t)-x^*)+t\dot{x}(t)}^2$.
We thus have:
\begin{equation*}
b(t_1) \leqslant \frac{\epsilon^2}{2 t_1^{2\lambda+1}}.
\end{equation*}
Since $\gamma \leqslant 2$, $\lambda =\frac{2 \alpha}{\gamma+2}$ and $p=\frac{2\gamma \f}{\gamma+2}-2$, we have
$
2\lambda+1
\geq 
p+1
$, and thus
\begin{equation*}
b(t_1) \leqslant \frac{\epsilon^2}{2 t_1^{p+1}}.
\end{equation*}


For $\epsilon =\sqrt{\frac{
\ell}{2}}$ for example, there exists thus some $t_1 >T$ such that $b(t_1) \leqslant \frac{\ell}{4 t_1^{p+1}}$.
 Then $a(t_1) \geqslant  \frac{\ell}{4 t_1^{p+1}}$, i.e.
 $F(x(t_1))-F^* \geqslant  \frac{\ell}{4 t_1^{p+2}}$.
 Since $p+2= \frac{2 \gamma \alpha}{\gamma +2}$, this concludes the proof.

\subsection{Proof of Theorem \ref{Theo2}}
We detail here the proof of Theorem \ref{Theo2}.

Let us consider $\gamma_1>2$, $\gamma_2 >2$, and $\f\geqslant  \frac{\gamma_1+2}{\gamma_1-2}$. We consider here functions $\HH$ for all $x^*$ in the set $X^*$ of minimizers of $F$ and prove that these functions are uniformely bounded. More precisely for any $x^*\in X^*$ we define $\HH(t)$ with $p=\frac{4}{\gamma_1-2}$ and $\lambda=\frac{2}{\gamma_1-2}$. 
With this choice of $\lambda$ and $p$, using Hypothesis $\hu(\gamma_1)$ we have from Lemma~\ref{LemmeFonda}:
\begin{equation*}
\HH'(t)\leqslant 2t^{\frac{4}{\gamma_1-2}}\left(\frac{\gamma_1+2}{\gamma_1-2}-\f\right)b(t).
\end{equation*} 
which is non-positive when $\f\geqslant\frac{\gamma_1+2}{\gamma_1-2}$, which implies that the function $\HH$ is bounded above. Hence for any choice of $x^*$ in the set of minimizers $X^*$, the function $\HH$ is bounded above and since the set of minimizers is bounded (F is coercive), there exists $A>0$ and $t_0$ such that for all choices of $x^*$ in $X^*$, 
\begin{equation*} \label{inegHHAtz}
\HH(t_0)\leqslant A,
\end{equation*}
which implies that for all $x^* \in X^{*}$ and
for all $t\geqslant  t_0$
\begin{equation*} \label{inegHHA}
\HH(t)\leqslant A.
\end{equation*}
Hence for all $t\geqslant  t_0$ and for all $x^*\in X^*$
\begin{equation*}\label{BoundWold}
t^{\frac{4}{\gamma_1-2}}t^2(F(x(t))-F^*)\leqslant \frac{\vert \xi\vert}{2}
t^{\frac{4}{\gamma_1-2}}\norm{x(t)-x^*}^2
+A,
\end{equation*} 
which implies that 
\begin{equation}\label{BoundW}
t^{\frac{4}{\gamma_1-2}}t^2(F(x(t))-F^*)\leqslant \frac{\vert \xi\vert}{2}
t^{\frac{4}{\gamma_1-2}}d(x(t),X^{*})^2
+A.
\end{equation} 

We now set:
\begin{equation}\label{defv}
v(t):=t^{\frac{4}{\gamma_2-2}}d(x(t),X^*)^2.
\end{equation}
Using  \eqref{BoundW} we have:
\begin{equation}\label{Boundu1bos}
t^{\frac{2 \gamma_1}{\gamma_1-2}}(F(x(t))-F^*)\leqslant \frac{\vert \xi\vert}{2}
t^{\frac{4}{\gamma_1-2}-\frac{4}{\gamma_2-2}} 
v(t)+A.
\end{equation}
Using the hypothesis $\hd(\gamma_2)$ applied under the form given by Lemma \ref{lem:H2} (since $X^*$ is compact), there exists $K>0$ such that
\begin{equation*}
K\left(t^{-\frac{4}{\gamma_2-2}}v(t)\right)^{\frac{\gamma_2}{2}}\leqslant F(x(t))-F^*,
\end{equation*}
which is equivalent to 
\begin{equation*}
Kv(t)^{\frac{\gamma_2}{2}}
t^{\frac{-2\gamma_2}{\gamma_2-2}}
\leqslant F(x(t))-F^*.
\end{equation*}
Hence:
\begin{equation*}
K
t^{\frac{2\gamma_1}{\gamma_1-2}} t^{\frac{-2\gamma_2}{\gamma_2-2}}
v(t)^{\frac{\gamma_2}{2}}\leqslant 
t^{\frac{2\gamma_1}{\gamma_1-2}} (F(x(t))-F^*).
\end{equation*}
Using \eqref{Boundu1bos}, we obtain: 
\begin{equation*}
K
t^{\frac{2\gamma_1}{\gamma_1-2}-\frac{2\gamma_2}{\gamma_2-2}}
v(t)^{\frac{\gamma_2}{2}}\leqslant \frac{|\xi|}{2} 
t^{\frac{4}{\gamma_1-2}-\frac{4}{\gamma_2-2}} 
v(t)+A,
\end{equation*}
i.e.:
\begin{equation}\label{vbounded}
K
v(t)^{\frac{\gamma_2}{2}}\leqslant \frac{|\xi|}{2} 
v(t)+A
t^{\frac{4}{\gamma_2-2}-\frac{4}{\gamma_1-2}}.
\end{equation}
Since $2<\gamma_1 \leqslant \gamma_2$, we deduce that $v$ is bounded.
Hence, using  \eqref{Boundu1bos} there exists some positive constant $B$ such that:
\begin{equation*}
F(x(t))-F^*\leqslant 
B t^{\frac{-2 \gamma_2}{\gamma_2-2}}
+A t^{\frac{-2 \gamma_1}{\gamma_1-2}}.
\end{equation*}
Since $2<\gamma_1 \leqslant  \gamma_2$, we have $\frac{-2 \gamma_2}{\gamma_2-2} \geqslant  \frac{-2 \gamma_1}{\gamma_1-2}$.
Hence we deduce that  $F(x(t))-F^* = O \left( t^{\frac{-2 \gamma_2}{\gamma_2-2}}\right)$.\\



\subsection{Proof of Corollary~\ref{Corol2}}
We are now in position to prove Corollary~\ref{Corol2}. The first point of Corollary~\ref{Corol2} is just a particular instance of Theorem~\ref{Theo2}.
In the sequel, we prove the second point of Corollary~\ref{Corol2}.

Let $t\geqslant  t_0$ and $\tilde x\in X^*$ such that 
\begin{equation*}
\|x(t)-\tilde x\|=d(x(t),X^*).
\end{equation*}
We previously proved that there exists $A>0$ such that for any $t\geqslant  t_0$ and any $x^*\in X^*$,
\begin{equation*}
\HH(t)\leqslant A.
\end{equation*}
For the choice $x^*=\tilde x$ this inequality ensures that 
\begin{equation*}
\frac{t^{\frac{4}{\gamma-2}}}{2}\norm{\lambda (x(t)-\tilde x)+t\dot x(t)}^2+t^{\frac{4}{\gamma-2}}\frac{\xi}{2}d(x(t),\tilde x)^2\leqslant A,
\end{equation*}
which is equivalent to 
\begin{equation*}
\frac{t^{\frac{4}{\gamma-2}}}{2}\norm{\lambda (x(t)-\tilde x)+t\dot x(t)}^2\leqslant \frac{|\xi|}{2}v(t)+A,
\end{equation*}
where $v(t)$ is defined in \eqref{defv} with $\gamma=\gamma_2$.
Using the fact that the function $v$ is bounded (a consequence of \eqref{vbounded}) we deduce that there exists a positive constant $A_1>0$ such that:
\begin{equation*}
\norm{\lambda (x(t)-\tilde x)+t\dot x(t)}\leqslant \frac{A_1}{t^{\frac{2}{\gamma-2}}}.
\end{equation*}
Thus:
\begin{equation*}
t\norm{\dot x(t)}\leqslant \frac{A_1}{t^{\frac{2}{\gamma-2}}}+|\lambda|d(x(t),\tilde x)=\frac{A_1+|\lambda|\sqrt{v(t)}}{t^{\frac{2}{\gamma-2}}}.
\end{equation*}
Using once again the fact that the function $v$ is bounded we deduce that there exists a real number $A_2$ such that 
\begin{equation*}
\norm{\dot x(t)}\leqslant \frac{A_2}{t^{\frac{\gamma}{\gamma-2}}},
\end{equation*}
which implies that $\norm{\dot x(t)}$ is an integrable function. As a consequence, we deduce that the trajectory $x(t)$ has a finite length.

\subsection{Proof of Proposition \ref{prop:gap}}
The idea of the proof is very similar to that of Proposition~\ref{prop_optimal} (optimality of the convergence rate in the sharp case i.e. when $\gamma \in (1,2]$).

For the exact same choice of parameters $p=\frac{2\gamma\alpha}{\gamma+2}-2$ and $\lambda=\frac{2\alpha}{\gamma+2}$ and assuming that $1+\frac{2}{\gamma}<\alpha < \frac{\gamma+2}{\gamma-2}$, we first show that the energy $\mathcal H$ is non-decreasing and then:
\begin{equation}
\forall t\geqslant t_0,~\mathcal H(t) \geqslant \ell,\label{infH}
\end{equation}
where: $\ell=\mathcal H(t_0)>0$. Indeed, since $\gamma>2$ and $\alpha <\frac{\gamma+2}{\gamma-2}$, a straightforward computation shows that: $\lambda^2 -|\xi|>0$, so that:
\begin{eqnarray*}
\mathcal H(t_0) 
&=& t_0^{p+2}|x(t_0)|^\gamma + \frac{t_0^p}{2}\left(|\lambda x(t_0)+t_0\dot x(t_0)|^2 -|\xi||x(t_0)|^2\right)\\
&=& t_0^{p+2}|x(t_0)|^\gamma + \frac{t_0^p}{2}\left( \lambda^2 -|\xi|\right)|x(t_0)|^2 >0,
\end{eqnarray*}
without any additional assumption on the initial time $t_0>0$. 

Let $T>t_0$. We set: $y(t)=t^\lambda x(t)$. If $y(t)$ is bounded as it is in Proposition~\ref{prop_optimal}, by the exact same arguments, we prove that there exists $t_1>T$ such that: $b(t_1) \leq \frac{\ell}{4t_1^{p+1}}$. Moreover since $\xi<0$ we deduce from \eqref{infH} that:
$$t_1^{p+1}(a(t_1)+b(t_1)) \geqslant \ell.$$Hence:
$$a(t_1)=t_1(F(x(t_1)-F^*) \geqslant \frac{\ell}{4t_1^{p+1}},$$
i.e.: $F(x(t_1))-F^* \geqslant \frac{\ell}{4t_1^{p+2}} = \frac{\ell}{4t_1^\frac{2\alpha\gamma}{\gamma+2}}$.

If $y(t)$ is not bounded, then the proof is even simpler: indeed, in that case, for any $K>0$, there exists $t_1\geqslant T$ such that: $y(t_1)\geq K$, hence:
$$F(x(t_1))-F^*=|x(t_1)|^\gamma \geqslant \frac{K}{t_1^{\lambda\gamma}}=\frac{K}{t_1^{\frac{2\alpha\gamma}{\gamma+2}}},$$
which concludes the proof.

\section*{Acknowledgement}
This study has been carried out with financial support from the French state, managed by the French National Research Agency (ANR GOTMI)  (ANR-16-VCE33-0010-01) and partially supported by ANR-11-LABX-0040-CIMI within the program
ANR-11-IDEX-0002-02. J.-F. Aujol is a member of Institut Universitaire de France.


\appendix

\section{Proof of Lemma~\ref{LemmeFonda}} \label{appendix}
We prove here Lemma~\ref{LemmeFonda}. Notice that the computations are standard (see e.g. \cite{AujolDossal}).



\begin{lemma} \label{lemma_tec}
\begin{eqnarray*}
\mathcal{E}'(t)&=&2a(t)+\lambda t\ps{-\nabla F(x(t))}{x(t)-x^*}
 + (\xi- \lambda (\lambda + 1 -\f)
)\ps{\dot{x}(t)}{x(t)-x^*}\\
&&+
2(\lambda+1-\f)
b(t)
-2\lambda^2(\lambda+1-\f) c(t)
\end{eqnarray*}\label{lem:tec1}
\end{lemma}

\begin{proof} Let us differentiate the energy $\mathcal E$:
\begin{eqnarray*}
\mathcal{E}'(t)&=&2a(t)+t^2\ps{\nabla F(x(t))}{\dot{x}(t)}+\ps{\lambda \dot{x}(t)+t\ddot{x}(t)+\dot{x}(t)}{\lambda(x(t)-x^*)+t\dot{x}(t)}\\
&&+\xi\ps{\dot{x}(t)}{x(t)-x^*}\vspace{.25cm}\\
&=&2a(t)+t^2\ps{\nabla F(x(t))+\ddot{x}(t)}{\dot{x}(t)}+(\lambda+1)t\norm{\dot{x}(t)}^2+\lambda t\ps{\ddot{x}(t)}{x(t)-x^*}\\
&&+(\lambda(\lambda+1)+\xi)\ps{\dot{x}(t)}{x(t)-x^*}\vspace{.25cm}\\
&=&2a(t)+t^2\ps{-\frac{\f}{t}\dot{x(t)}}{\dot{x}(t)}+(\lambda+1)t\norm{\dot{x}(t)}^2+\lambda t\ps{\ddot{x}(t)}{x(t)-x^*}\\
&&+(\lambda(\lambda+1)+\xi)\ps{\dot{x}(t)}{x(t)-x^*}\\
&=&2a(t)+t(\lambda+1-\f)\norm{\dot{x}(t)}^2
+\lambda t\ps{\ddot{x}(t)}{x(t)-x^*}+
(\lambda(\lambda+1)+\xi)\ps{\dot{x}(t)}{x(t)-x^*}.
\end{eqnarray*}
Using the ODE \eqref{ODE}, we get:
\begin{eqnarray*}
\mathcal{E}'(t)&=&2a(t)+t(\lambda+1-\f)\norm{\dot{x}(t)}^2 +\lambda t\ps{-\nabla F(x(t))-\frac{\f}{t}\dot{x(t)}}{x(t)-x^*}\\
&&+ (\lambda(\lambda+1)+\xi)\ps{\dot{x}(t)}{x(t)-x^*}\\
&=&2a(t)+t(\lambda+1-\f)\norm{\dot{x}(t)}^2+\lambda t\ps{-\nabla F(x(t))}{x(t)-x^*}\\
&&+ (\lambda(\lambda+1)-\f\lambda+\xi)\ps{\dot{x}(t)}{x(t)-x^*}.
\end{eqnarray*}
Observing now that:
\begin{equation*}\label{eqReformEner}
\frac{1}{t}\norm{\lambda(x(t)-x^*)+t\dot{x}(t)}^2=t\norm{\dot{x}(t)}^2+2\lambda\ps{\dot{x}(t)}{x(t)-x^*}+\frac{\lambda^2}{t}\norm{x(t)-x^*}^2,
\end{equation*}
we can write:
\begin{eqnarray*}
\mathcal{E}'(t)&=&2a(t)+\lambda t\ps{-\nabla F(x(t))}{x(t)-x^*}  + (\xi- \lambda (\lambda + 1 -\f)
)\ps{\dot{x}(t)}{x(t)-x^*}
\\
&&+
(\lambda+1-\f)\frac{1}{t}\norm{\lambda(x(t)-x^*)+t\dot{x}(t)}^2-\frac{\lambda^2(\lambda+1-\f)}{t}\norm{x(t)-x^*}^2.
\end{eqnarray*}

\end{proof}


\begin{corollary}\label{lemma_tec_conv_beta2}
 If $F$ satisfies the hypothesis $\hu(\gamma)$ 
 and
if $\xi = \lambda (\lambda + 1 -\f)$, then:
 \begin{align}\label{EqConvBeta}
 \mathcal{E}'(t)\leqslant&
 (2- \gamma \lambda) a(t)
+2(\lambda+1-\f)
b(t)
-2\lambda^2(\lambda+1-\f) c(t)
\end{align}
\end{corollary}

\begin{proof}
Choosing $\xi = \lambda (\lambda + 1 -\f)$ in Lemma~\ref{lem:tec1}, we get:
$$\mathcal{E}'(t)=2a(t)+\lambda t\ps{-\nabla F(x(t))}{x(t)-x^*} 
+2(\lambda+1-\f)
b(t)
-2\lambda^2(\lambda+1-\f) c(t).
$$
Applying now the assumption $\hu(\gamma)$, we finally obtain the expected result.

\end{proof}

One can notice that if $F(x)=|x|^{\gamma}$ the inequality of Lemma
\ref{lem:tec1} is actually an equality when $\xi = \lambda (\lambda + 1 -\f)$. This ensures that for this specific function $F$, the inequality in Lemma~\ref{LemmeFonda} is an equality. 

\begin{lemma} \label{lemma_tec4}
If $F(x)=|x|^{\gamma}$ and if $\xi = \lambda (\lambda + 1 -\f)$, then
\begin{eqnarray*}
\mathcal{H}'(t)&=&
t^{p} \left[ (2+p) a(t)+\lambda t\ps{-\nabla F(x(t))}{x(t)-x^*} 
\right.
+(2 \lambda+2 -2\f +p)b(t)
\\ &&
\left.
~~~+\lambda(\lambda+1-\f) (-2\lambda +p)c(t)
\right]
\end{eqnarray*}
\end{lemma}

\begin{proof}
We have $\mathcal{H}(t)=t^p \mathcal{E}(t)$. Hence $\mathcal{H}'(t)=t^p \mathcal{E}'(t)+pt^{p-1}\mathcal{E}(t)=t^{p-1} (t \mathcal{E}'(t)+p\mathcal{E}(t))$. We conclude by using Lemma~\ref{lem:tec1}.
\end{proof}

In conclusion, to prove Lemma \ref{LemmeFonda}, it is sufficient to plug the assumption $\hu(\gamma)$ into the equality of Lemma \ref{lemma_tec4}.   



\bibliographystyle{siamplain}
\bibliography{reference}
\end{document}